\documentclass[12pt,amstex]{amsart}

\usepackage{epsfig}
\usepackage{amsmath}
\usepackage{amssymb}
\usepackage{amscd}
\usepackage{graphicx}

\usepackage{epsfig}
\usepackage{float}

\usepackage{calc}         
\usepackage{graphicx}     
\usepackage{ifthen}       
\usepackage{subfigure}    
\usepackage{pst-all}      
\usepackage{pst-poly}     
\usepackage{multido}

\newtheorem{thm}{Theorem}[section]
\newtheorem{lem}[thm]{Lemma}

\newtheorem{cor}[thm]{Corollary}
\newtheorem{ex}[thm]{Example}

\theoremstyle{definition}

\theoremstyle{remark}
\newtheorem{rem}[thm]{Remark}

\numberwithin{equation}{section}

\def \U {\mathcal{U}}

\def \F {\mathcal{F}}

\def \p1 {\pi _1}
\def \p {\mathbb{P}}

\begin{document}
\title{topological pressure for sub-additive potentials of
amenable group actions}

\date{Apr. 3, 2011}

\author[B. Liang]{Bingbing Liang}
\address{Department of Mathematics, University of Science and Technology of
China, Hefei Anhui 230026, P. R. China}
\email{bbliang@mail.ustc.edu.cn}

\author[K. Yan]{Kesong Yan}
\address{Department of Mathematics, University of science and Technology of
China, Hefei Anhui 230026, P. R. China, and Department of
Mathematics and Computer Science, Liuzhou Teachers College, Liuzhou,
Guangxi, 545004, P.R China} \email{ksyan@mail.ustc.edu.cn}




\begin{abstract}
The topological pressure for any sub-additive potentials of a
countable discrete amenable group action and any given open cover is
defined. A local variational principle for the topological pressure
is established.
\end{abstract}

\maketitle


\section{Introduction and main result}

Entropies are fundamental to our current understanding of dynamical
systems. The classical measure-theoretic entropy for an invariant
measure and the topological entropy were introduced in \cite{Kol}
and \cite{AKM} respectively, and the classical variational principle
was completed in \cite{Goo69, Goo71}. Since then a subject involving
to define new measure-theoretic and topological notations of entropy
and study the relationship between them has gained a lot of
attention in the study of dynamical systems.

Topological pressure is a generalization of topological entropy for
a dynamical system. The notion was first introduced by Ruelle
\cite{Rue} in 1973 for an expansive dynamical system and later by
Walters \cite{Wal75} for the general case. The variational principle
formulated by Walters can be stated precisely as follows: Let $(X,
T)$ be a topological dynamical system, where $X$ is a compact metric
space and $T: X \rightarrow X$ is a continuous map, and $f: X
\rightarrow {\mathbb R}$ is a continuous function. Let $P(T, f)$
denote the topological pressure of $f$ (see \cite{Wal}). Then
\begin{equation} \label{eq:1-1}
P(T, f)=\sup \left\{h_{\mu}(T) + \int \, f \, \mathrm{d}\mu: \mu \in
{\mathcal M}(X, T)\right\},
 \end{equation}
where $\mathcal{M}(X, T)$ denotes the spaces of all $T$-invariant
Borel probability measures on $X$ and $h_{\mu}(T)$ denotes the
measure-theoretic entropy of $\mu$.

The theory related to the topological pressure, variational
principle and equilibrium states plays a fundamental role in
statistical mechanics, ergodic theory and dynamical systems (see,
e.g., the books \cite{Bow, Kel, Rue2, Wal}). Since the works of
Bowen \cite{Bow2} and Ruelle \cite{Rue3}, the topological pressure
has become a basic tool in the dimension theory related to dynamical
systems. In 1984, Pesin and Pitskel \cite{PP} defined the
topological pressure of additive potentials for non-compact subsets
of compact metric spaces and proved the variational principle under
some supplementary conditions. In 1988, Falconer \cite{Fal}
considered the thermodynamic formalism for sub-additive potentials
for mixing repellers. He proved the variational principle for the
topological pressure under some Lipschitz conditions and bounded
distortion assumptions on the sub-additive potentials. In 1996,
Barreira \cite{Bar} extended the work of Pesin and Pitskel. He
defined the topological pressure for an arbitrary sequence of
continuous functions on an arbitrary subset of compact metric
spaces, and proved the variational principle under a strong
convergence assumption on the potentials. In 2008, Y. Cao, D. Feng
and W. Huang \cite{CFH} generalized Ruelle and Walters's results to
sub-additive potentials in general compact dynamical systems.

Since notions of entropy pairs were introduced in both topological
\cite{Bla93} and measure-theoretic system \cite{BHM}, much attention
has been paid to the study the local variational principle of
entropy. Recently, Kerr and Li introduced various notions of
independence and give a uniform treatment of entropy pairs and
sequence entropy pairs \cite{KL, KL1}. An overview of local entropy
theory can see the survey paper \cite{GY}. In 2007, to study the
local variational principle of topological pressure, W. Huang and Y.
Yi \cite{HYi} introduced a new definition of topological pressure
for open covers. They proved a local variational principle for
topological pressure for any given open cover.

In this paper, we generalize Huang-Yi's results to dynamical systems
acting by a countable discrete amenable group. Let $(X, G)$ be an
amenable group action dynamical system. We define the local
topological pressure for sub-additive potentials
$\mathcal{F}=\{f_E\}_{E \in {\mathcal F}(G)}$ and set up a local
variational principle between the topological pressure and
measure-theoretical entropies.

Now we formulate our results.  Throughout the paper, we let $(X, G)$
be a $G$-system, where $G$ is a countable discrete amenable group
and $X$ is a compact metric space. A {\bf sub-additive potential on
$(X, G)$} is a collection ${\mathcal F}=\{f_E\}_{E \in {\mathcal
F}(G)}$ of continuous real valued function on $X$ satisfies the
following conditions:

\vspace{1mm}

{\bf (C1)}\ $f_{E \cup F}(x) \leq f_E(x) + f_F(x)$ for all $x \in X$
and all $E \cap F =\emptyset, E, F \in {\mathcal F}(G)$;

\vspace{1mm}

{\bf (C2)}\ $f_{Eg}(x)=f_E(gx)$ for all $x \in X$, $g \in G$ and $E
\in {\mathcal F}(G)$;

\vspace{1mm}

{\bf (C3)}\ $C=\sup \limits_{E\in F(X)}\sup\limits_{x\in X,g\in
G}(f_E(x)-f_{E \cup \{g\}} (x))<\infty$.
\vspace{2mm}
\\
As a main result, we obtain the following local variational
principle.

\begin{thm}
{\bf (Local variational principle)} Let $(X, G)$ be a $G$-system,
$\mathcal{U}\in \mathcal{C}_X^o$ and ${\mathcal F}=\{f_E\}_{E \in
{\mathcal F}(G)}$ a sub-additive potential on $(X, G)$. Then

\begin{equation} \label{eq:main}
P(G,\mathcal {F};\mathcal {U})=\sup_{\mu\in
\mathcal{M}(X,G)}\left\{h_\mu(G,\mathcal{U})+\mathcal{F}_*(\mu)
\right\}
 \end{equation}
and the supremum can be attained in $\mathcal{M}^e(X, G)$, if one of
the following conditions holds:

{\rm (1)} $G$ is an Abelian group;

{\rm (2)} ${\mathcal F}$ is strongly sub-additive, i.e. $f_{E\cup
F}+f_{E \cap F}\leq f_E+f_F$ for all $E,F \in F(G)$.
\end{thm}

\vspace{4mm}

In particular, if $f$ is a continuous function on $X$, then
$$\mathcal{F}=\left\{f_E=\sum_{g \in E}f \circ g: E \in {\mathcal F}(G)\right\}$$
satisfies the condition (2) in Theorem 1.1. In this case, write
$P(G,f;\mathcal {U})=P(G,\mathcal {F};\mathcal {U})$, we can get

\begin{cor}
Let $(X, G)$ be a $G$-system, $\mathcal{U}\in \mathcal{C}_X^o$ and
$f \in C(X)$, then
$$P(G,f;\mathcal {U})=\sup_{\mu\in \mathcal{M}(X,G)}
\left\{h_\mu(G,\mathcal{U})+ \int_X \, f \, \mathrm{d}\mu\right\}$$
and the supremum can be attained in $\mathcal{M}^e(X, G)$.
\end{cor}

\vspace{4mm}

The paper is organized as follows: in section 2, we recall some
knowledge about amenable group, and the definition and basic
properties of local measure-theoretic entropy for amenable group
action. Moreover, we introduce the local pressure for a sub-additive
potential. In section 3, we provide some useful lemmas and prove
Theorem 1.1. In section 4, we give a nontrivial example of
sub-additive potential.

\section{Pressure of an Amenable Group Action}

\subsection{Backgrounds of a countable discrete amenable group} Let
$G$ be a countable discrete infinite group and $\mathcal{F}(G)$ the
set of all finite non-empty subsets of $G$. A {\bf tile} $T
\subseteq G$ is a finite subset that has a collection of right
translates that partitions $G$, i.e., there is a set $C \subseteq G$
of {\bf tiling centers} such that $\{Tc: c \in C\}$ form a disjoint
family whose union $TC$ is all of $G$. Note that $T \in
\mathcal{F}(G)$ is a tile of $G$ if and only if any $A \in
\mathcal{F}(G)$ can be covered by disjoint right translates of $T$.

A group $G$ is said to be {\bf amenable} if for all $\epsilon>0$ and
all $K \in \mathcal{F}(G)$, there exists $F \in \mathcal{F}(G)$ such
that
$$\frac{\left|F \bigtriangleup KF\right|}{\left|F\right|}<\epsilon.$$
Observe that a countable group is amenable if and only if there is a
sequence $\{F_n\}_{n \in {\mathbb N}} \subseteq \mathcal{F}(G)$ such
that
$$
\lim_{n \rightarrow \infty}\frac{\left|KF_n \bigtriangleup
F_n\right|}{\left|F_n\right|}=0
$$
for all $K \in \mathcal{F}(G)$. Such a sequence is called a {\bf
F$\phi$lner sequence} of $G$ (see \cite{E-F}). For a more complete
description of this class of groups see, for example, \cite{Gre} or
\cite{Pat}.

It is well known that the class of amenable groups contains all
finite groups, Abelian groups, it is closed by taking subgroups,
quotients, extensions and inductive limits. All finitely generated
groups of subexponential growth are amenable. A basic example of a
nonamenable group is the free group of rank 2.

Cyclic groups have F$\phi$lner sequences of tiling sets, and it can
build up form them to show that all solvable groups, finite
extensions thereof, increasing unions, etc., in brief the so-called
class of {\bf elementary amenable groups}, all have F$\phi$lner
sequences of tiling sets. In particular, all Abelian group have
tiling F$\phi$lner sequences. Unfortunately it is an open problem
that whether all countable discrete amenable groups have F$\phi$lner
sequences of tiling sets \cite{OW}.

Let $f: \mathcal{F}(G) \rightarrow {\mathbb R}$ be a function. We
say that $f$ is

(i) {\bf monotone}, if $f(E) \leq f(F)$ for any $E, F \in
\mathcal{F}(G)$ with $E \subseteq F$;

(ii) {\bf non-negative}, if $f(F) \geq 0$ for any $F \in
\mathcal{F}(G)$;

(iii) {\bf $\mathbf{G}$-invariant}, if $f(Fg)=f(F)$ for any $F \in
\mathcal{F}(G)$ and $g \in G$;

(iv) {\bf sub-additive}, if $f(E \cup F) \leq f(E) + f(F)$ for any
$E, F \in \mathcal{F}(G)$;

(v) {\bf strongly sub-additive}, if $f(E \cup F) + f(E \cap F) \leq
f(E) + f(F)$ for any $E, F \in \mathcal{F}(G)$.
\\

The following limit theorem for invariant sub-additive functions on
finite subsets of amenable groups is due to Ornstein and Weiss (see
\cite{Gro, LiW, OW}). It plays a central role in the definition of
some dynamical invariants such as topological entropy and
measure-theoretic entropy.

\begin{lem} {\bf (Ornstein-Weiss)}
Let $G$ be a countable amenable group. Let $f: \mathcal{F}(G)
\rightarrow {\mathbb R}$ be a monotone non-negative $G$-invariant
sub-additive function. Then there is a real number
$\lambda=\lambda(G, f) \geqslant 0$ dependent only on $G$ and $f$
such that
$$ \lim_{n \rightarrow \infty}\frac{f(F_n)}{|F_n|}=\lambda$$
for all F$\phi$lner sequence $\{F_n\}_{n \in {\mathbb N}}$ of $G$.
\end{lem}

\begin{rem}
{\rm (1)} If $f$ is also strongly sub-additive, then
$$\lim_{n\rightarrow \infty} \frac{f(F_n)}{|F_n|}=\inf_{F\in
F(G)}\frac{f(F)}{|F|}.$$

{\rm (2)} If $G$ admits a F$\phi$lner sequence $\{F_n\}_{n \in
{\mathbb N}}$ of tiling sets, then
$$\lim_{n \rightarrow \infty}\frac{f(F_n)}{|F_n|}=
\inf_{n \in {\mathbb N}}\frac{f(F_n)}{|F_n|},$$ and the value of the
limits is independent of the choice of such a F$\phi$lner sequence.
For details can see \cite{Wei}.
\end{rem}

\subsection{Topological pressure for sub-additive potentials} Let $(X,G)$ be a
$G$-system. Denote by ${\mathcal B}_X$ the collection of all Borel
subsets of $X$. Recall that a {\bf cover} of $X$ is a family of
Borel subsets of $X$ whose union is $X$. An {\bf open cover} is one
that consists of open sets. A {\bf partition} of $X$ is a cover of
$X$ consisting of pairwise disjoint sets. We denote the set of
finite covers, finite open covers and finite partition of $X$ by
${\mathcal C}_X$, ${\mathcal C}_X^o$ and ${\mathcal P}_X$,
respectively. Given two covers ${\mathcal U}, {\mathcal V} \in
{\mathcal C}_X$, ${\mathcal U}$ is said to be {\bf finer} than
${\mathcal V}$ (denoted by ${\mathcal U}\succeq {\mathcal V}$) if
each element of ${\mathcal U}$ is contained in some element of
${\mathcal V}$. Let ${\mathcal U} \vee {\mathcal V}=\{U \cap V : U
\in {\mathcal U}, V \in {\mathcal V}\}$. Given $F \in {\mathcal
F}(G)$ and ${\mathcal U} \in {\mathcal C}_X$, set ${\mathcal
U}_F=\bigvee_{g \in F}g^{-1}{\mathcal U}$ (letting ${\mathcal
U}_{\emptyset}=\{X\}$).

We now define the topological pressure of sub-additive potential
${\mathcal F}$ relative an open cover. For $E \in {\mathcal F}(G)$
and ${\mathcal U} \in {\mathcal C}_X^o$, we define
$$P_E(G, {\mathcal F}; {\mathcal U}):=\inf \left\{\sum_{V \in {\mathcal V}}
\sup_{x \in V} \mathrm{e}^{f_E(x)}: {\mathcal V} \in {\mathcal C}_X
\mbox{\ and } {\mathcal V} \succeq {\mathcal U}_E\right\}.$$ For
${\mathcal V} \in {\mathcal C}_X$, we let $\alpha$ be the Borel
partition generated by ${\mathcal V}$ and define
\begin{equation} \label{eq:1}
{\mathcal P}^*(\mathcal V)=\left\{\beta \in {\mathcal P}_X:
\begin{array}{l l}
& \beta \succeq {\mathcal V} \mbox{\ and each atom of } \beta
\mbox{\ is the union of }\\
& \mbox{\ some atoms of } \alpha.
\end{array}
\right\}.
\end{equation}
Note that ${\mathcal P}^*({\mathcal V})$ is a finite set. Following
the idea of Huang-Yi (see \cite[Lemma 2.1]{HYi}), we have the
following result.

\begin{lem}
For $E \in {\mathcal F}(G)$ and ${\mathcal U} \in {\mathcal C}_X$,
we have
\begin{equation} \label{eq:2}
P_E(T, {\mathcal F};{\mathcal U})=\min \left\{\sum_{B \in \beta}
\sup_{x \in B} \mathrm{e}^{f_E(x)}: \beta \in {\mathcal
P}^{*}({\mathcal U}_E)\right\}.
\end{equation}
\end{lem}

\begin{lem}
The following hold:

{\rm (1)}\ $K=\sup \left\{\frac{|f_E(x)|}{|E|}: x \in X, E\in
F(G)\right\}<\infty$;

{\rm (2)}\ Set $\mathcal{G}=\{f_E(x)+ C|E|: E\in F(G), f_E\in
\mathcal {F}\}$, then ${\mathcal {G}}$ is a monotone non-negative
sub-additive function. If $\mathcal {F}$ is strongly sub-additive,
then $\mathcal {G}$ is also strongly sub-additive.
\end{lem}

\begin{proof}
It easily follows from conditions (C1), (C2) and (C3).
\end{proof}

It is not hard to see that $E \in {\mathcal F}(G) \mapsto \log
P_E(G, {\mathcal G}, {\mathcal U})$ is a monotone non-negative
$G$-invariant sub-additive function. By Lemma 2.1,
$$\lim_{n \rightarrow \infty}
\frac{1}{|F_n|}\log P_{F_n}(G, {\mathcal F}; {\mathcal U})=\lim_{n
\rightarrow \infty} \frac{1}{|F_n|}\log P_{F_n}(G, {\mathcal G};
{\mathcal U})-C$$ is independence of the choice of the F$\phi$lner
sequence $\{F_n\}_{n \in \mathbb{N}}$. Define the {\bf topological
pressure of ${\mathcal F}$ relative to ${\mathcal U}$} as
\begin{equation} \label{eq:pressure}
P(G, {\mathcal F}; {\mathcal U}):=\lim_{n \rightarrow \infty}
\frac{1}{|F_n|}\log P_{F_n}(G, {\mathcal F}; {\mathcal U}),
 \end{equation}
where $\{F_n\}_{n \in {\mathbb N}}$ is a F$\phi$lner sequence of
$G$. The {\bf topological pressure of ${\mathcal F}$} is defined by
\begin{equation} \label{eq:pressure2}
P(G, {\mathcal F}):=\sup_{\mathcal{U} \in \mathcal{C}_X^o} P(G,
{\mathcal F}; {\mathcal U})
 \end{equation}

For a $G$-invariant Borel probability measure $\mu$, denote
$${\mathcal F}_*(\mu):=\lim_{n \rightarrow \infty}\frac{1}{|F_n|}
\int \, f_{F_n} \, \mathrm{d} \mu,$$ where $\{F_n\}_{n \in {\mathbb
N}}$ is a F$\phi$lner sequence. The existence of the above limit
follows from conditions (C1) and (C2). We call ${\mathcal F}_*(\mu)$
the {\bf Lyapunov exponent of ${\mathcal F}$ with respect to $\mu$}.

\subsection{Measure-theoretic entropy} Recall the basic definitions
(see \cite{HYZ} for details). Let ${\mathcal M}(X)$, ${\mathcal
M}(X, G)$ and ${\mathcal M}^e(X, G)$ be the sets of all Borel
probability measures, $G$-invariant Borel probability measures on
$X$ and $G$-invariant ergodic measures, on $X$, respectively. Note
that amenability of $G$ ensures that ${\mathcal M}(X, G) \neq
\emptyset$ and both ${\mathcal M}(X)$ and ${\mathcal M}(X, G)$ are
convex compact metric spaces when endowed with the
weak$^*$-topology; ${\mathcal M}^e(X, G)$ is a $G_{\delta}$ subset
of $\mathcal{M}(X, G)$.

Given $\alpha, \beta \in {\mathcal P}_X$ and $\mu \in {\mathcal
M}(X)$, define
$$H_{\mu}(\alpha)=\sum_{A \in \alpha}-\mu(A) \log \mu(A) \mbox{\ \ \ and \ \ }
H_{\mu}(\alpha|\beta)=H_{\mu}(\alpha \vee \beta)-H_{\mu}(\beta).$$
One standard fact is that $H_{\mu}(\alpha|\beta)$ increases with
respect to $\alpha$ and decreases with respect to $\beta$. When $\mu
\in {\mathcal M}(X, G)$, it is not hard to see that $F \in {\mathcal
F}(G) \mapsto H_{\mu}(\alpha_F)$ is a monotone non-negative
$G$-invariant sub-additive function for a given $\alpha \in
{\mathcal P}_X$. The {\bf measure-theoretic entropy of $\mu$
relative to $\alpha$} is defined by
\begin{equation} \label{eq:entropy1}
h_{\mu}(G, \alpha)=\lim_{n \rightarrow \infty}
\frac{1}{|F_n|}H_{\mu}(\alpha_{F_n})=\inf_{F \in
F(G)}\frac{1}{|F|}H_{\mu}(\alpha_F),
 \end{equation}
where $\{F_n\}_{n \in {\mathbb N}}$ is a F$\phi$lner sequence of
$G$. The last identity follows from the fact that
$H_{\mu}(\alpha_F)$ is strongly sub-additive (see \cite[Lemma
3.1]{HYZ}). The {\bf measure-theoretic entropy of $\mu$} is defined
by
\begin{equation} \label{eq:entropy2}
h_{\mu}(G, X)=\sup_{\alpha \in {\mathcal P}_X} h_{\mu}(G, \alpha).
 \end{equation}

For a given ${\mathcal U} \in {\mathcal C}_X$, W. Huang, X. Ye and
G. Zhang (see \cite{HYZ}) introduced the following two types of {\bf
measure-theoretic entropies relative to ${\mathcal U}$} as
$$h_{\mu}^-(G, \mathcal{U}):=\lim_{n \rightarrow \infty}
\frac{1}{|F_n|}H_{\mu}(\mathcal{U}_{F_n}) \mbox{\ \ \ \ and \ \ \
}h_{\mu}^+(G, {\mathcal U}):=\inf_{\alpha \succeq {\mathcal U},
\alpha \in {\mathcal P}_X}h_{\mu}(G, \alpha),$$ where
$$H_{\mu}(\mathcal{U}):=\inf_{\alpha \succeq {\mathcal U}, \alpha
\in {\mathcal P}_X} H_{\mu}(\alpha).$$

\begin{rem}
(1) It is not hard to see that $h_{\mu}^-(G, {\mathcal U}) \leq
h_{\mu}^+(G, {\mathcal U})$. Moreover, Huang-Ye-Zhang (see
\cite[Theorem 4.14]{HYZ}) proved those two kinds of
measure-theoretic entropy are equivalent, thus, we denote by
$$h_{\mu}(G, \mathcal{U})=h_{\mu}^{\pm}(G, \mathcal{U}).$$

(2) For $\mu \in \mathcal{M}(X, G)$, the following holds (see
\cite{HYZ}): $$h_{\mu}(G, X)=\sup_{\mathcal{U} \in \mathcal{C}_X^o}
h_{\mu}(G, \mathcal{U}).$$
\end{rem}

\begin{lem}
{\bf (Ergodic decomposition of local entropy, \cite{HYZ})} Let
$\mathcal{U} \in \mathcal{C}_X^o$ and $\mu \in \mathcal{M}(X, G)$.
The local entropy function $h_{\cdot}(G, \mathcal{U})$ is upper
semi-continuous and affine on $\mathcal{M}(X, G)$, and
$$h_{\mu}(G, {\mathcal U})=\int_{\mathcal{M}^e(X, G)} \,
h_{\theta}(G, {\mathcal U}) \, \mathrm{d}m(\theta),$$ where
$\mu=\displaystyle{\int_{\mathcal{M}^e(X, G)} \, \theta \,
\mathrm{d}m(\theta)}$ is the ergodic decomposition of $\mu$.
\end{lem}

\section{A local variational principle of topological pressure}

In this section, we mainly prove a local variational principle of
topological pressure for sub-addition potentials.

\subsection{Some Lemmas} Now we give some lemmas which are needed in
our proof of Theorem 1.1. The first lemma is an obvious fact and we
omit the detailed proof.

\begin{lem}
Let $T \in  {\mathcal F}(G)$ be a tile of $G$ and $\{F_n\}_{n\in
\mathbb{N}}$ a F$\o$lner sequence. For each $n \in {\mathbb N}$, let
$C_n$ be the tiling center of $F_n$ relative to $T$, i.e., $F_n
\subseteq \bigsqcup_{c\in C_n}Tc$ and $Tc\cap F_n\neq \emptyset$ for
all $c\in C_n$, then
$$\lim_{n\rightarrow\infty}\frac{|TC_n|}{|F_n|}=1.$$
\end{lem}

\begin{lem}
Let $f: {\mathcal F}(G)\rightarrow \mathbb{R}$ be a non-negative
monotone strongly sub-additive function, $m, k \in \mathbb{N}$, $E,
F, B, E_1, \cdots, E_k \in {\mathcal F}(G)$. Then

\vspace{2mm}

{\rm (1)}  If $1_E(g)=\frac{1}{m}\sum_{i=1}^k 1_{E_i}(g)$ holds for
each $g\in G$, then $f(E)\leq \frac{1}{m}\sum_{i=1}^k f(E_i)$;

\vspace{2mm}

{\rm (2)}  If $K=\sup\{\frac{f(E)}{|E|}: E \in {\mathcal
F}(G)\}<\infty$, then
$$f(F)\leq \sum_{g\in F}\frac{1}{|B|}f(Bg)+ K \cdot |F\setminus A_{F,B}|\text{,}$$
where, $A_{F,B}=\{g\in G: B^{-1}g\subseteq F\}$.
\end{lem}

\begin{proof}
(1) Clearly, $\bigcup_{i=1}^k E_i=E$. Set
$\{A_1,\cdots,A_n\}=\bigvee_{i=1}^k \{E_i, E \setminus E_i\}$
(neglecting all empty elements). Set $K_0=\emptyset$,
$K_i=\bigcup_{j=1}^i A_j$,\ $i=1,\cdots,n$. Then
$\emptyset=K_0\subsetneq K_1\cdots \subsetneq K_n=E$. Note that if
for some $i=1,\cdots,n$ and  $j=1,\cdots,k$ with
$E_j\cap(K_i\setminus K_{i-1})\neq \emptyset$, then $K_i\setminus
K_{i-1}\subseteq E_j$, and so $K_i=K_{i-1}\cup (K_i\cap E_j)$. By
strongly sub-additive of $f$, we have $f(K_i)+f(K_{i-1}\cap E_j)\leq
f(K_{i-1})+f(K_i\cap E_j)$, i.e.,
$$f(K_i)-f(K_{i-1})\leq f(K_i\cap E_j)-f(K_{i-1}\cap E_j).$$
Now for each $i=1,\cdots,n$, we pick $k_i\in K_i\setminus K_{i-1}$,
one has
\begin{eqnarray*}
f(E)&=&\sum_{i=1}^n\left (\frac{1}{m}\sum_{i=1}^k
1_{E_i}(k_i)\right )(f(K_i)-f(K_{i-1}))\\
&=&\frac{1}{m}\sum_{j=1}^k\sum_{1\leq i\leq n \atop k_i\in
E_j}(f(K_i)-f(K_{i-1}))\\
&\leq&\frac{1}{m}\sum_{j=1}^k\sum_{1\leq i\leq n \atop k_i\in
E_j}(f(K_i\cap E_j)-f(K_{i-1}\cap E_j))\\
&\leq&\frac{1}{m}\sum_{j=1}^k\sum_{i=1}^n(f(K_i\cap
E_j)-f(K_{i-1}\cap E_j))=\frac{1}{m}\sum_{j=1}^kf(E_j),
\end{eqnarray*}

(2) Note that for each $l\in G$, we have $1_{\{h\in BF:
B^{-1}h\subseteq F\}}(l)=\frac{1}{|B|}\sum_{g\in F} 1_{\{h\in Bg:
B^{-1}h\subseteq F\}}(l)$. Using (1), we can get
$$f\left(\left\{h\in BF: B^{-1}h\subseteq
F \right\}\right)\leq \frac{1}{|B|}\sum_{g\in F} f\left(\left\{h\in
Bg: B^{-1}h\subseteq F \right\}\right)\leq \frac{1}{|B|}\sum_{g\in
F} f(Bg),$$ which implies
\begin{eqnarray*}
f(F)&\leq& f\left(\left\{h\in BF: B^{-1}h\subseteq F\right\}\right)+
f(F\setminus \{h\in BF:
B^{-1}h\subseteq F\})\\
&\leq&\frac{1}{|B|}\sum_{g\in F} f(Bg)+|F\setminus \{h\in BF:
B^{-1}h\subseteq F\}|\cdot K\\
&=&\frac{1}{|B|}\sum_{g\in F} f(Bg)+ K \cdot |F\setminus A_{F,B}|.
\end{eqnarray*}
The lemma is proved.
\end{proof}

\begin{lem}
Let $(X, G)$ be a zero-dimensional $G$-system, $\mu \in
\mathcal{M}(X, G)$, $\mathcal{F}$ is a sub-additive potential and
$\mathcal{U} \in \mathcal{C}_X^o$. Assume that for some $K \in
{\mathbb N}$, $\{\alpha_l\}_{l=1}^K$ is a sequence of finite clopen
(close and open) partitions of $X$ which are finer that ${\mathcal
U}$. Then for each $E \in {\mathcal F}(G)$, there is a finite subset
$B_E$ of $X$ such that each atom of $(\alpha_l)_E$, $l=1, \cdots,
K$, contains at most one point of $B_E$, and $\sum_{x \in B_E}
\mathrm{e}^{f_E(x)} \geq \frac{P_E(G, {\mathcal F},
\mathcal{U})}{K}$.
\end{lem}

\begin{proof}
The proof follows completely from that of \cite[Lemma 4.4]{HYi} and
is omitted.
\end{proof}

Let $(X, G)$ and $(Y, G)$ be two $G$-systems. A  continuous map
$\pi: X \rightarrow Y$ is called a {\bf homomorphism} or a {\bf
factor map} from $(X, G)$ to $(Y, G)$ if it is onto and $\pi \circ
g=g \circ \pi$ for all $g \in G$. $(X, G)$ is called an {\bf
extension} of $(Y, G)$ and $(Y, G)$ is called a {\bf factor} of $(X,
G)$. If $\pi$ is also injective then it is called an {\bf
isomorphism}.

\begin{lem}
Let $\pi :(X,G) \rightarrow (Y,G)$ be a factor map,
$\mathcal{F}=\{f_E\}_{E \in \mathcal{F}(G)} \subseteq C(Y)$
satisfies conditions (C1), (C2) and (C3). Let $\mu \in {\mathcal
M}(X, G)$, $\nu=\pi \mu$, $\alpha \in {\mathcal P}_Y$ and $\mathcal
{U}\in \mathcal {C}_Y^{o}$. Then

{\rm (1)} $h_{\mu}(G, \pi^{-1}(\alpha))=h_{\nu}(G, \alpha)$;

{\rm (2)} $P(G,\mathcal {F}\circ \pi; \pi^{-1}\mathcal {U})=
P(G,\mathcal {F};\mathcal {U})$;
\end{lem}

\begin{proof}
(1) It is an obvious fact.

(2) Let $\{F_n\}_{n \in {\mathbb N}}$ be a F$\phi$lner sequence of
$G$. Fix an $n \in {\mathbb N}$. If $\mathcal{V} \in \mathcal{C}_Y$
with $\mathcal{V} \succeq \U_{F_n}$, then $\pi^{-1}\mathcal{V} \in
\mathcal{C}_X$ and $\pi^{-1}\mathcal{V} \succeq (\pi^{-1}\U)_{F_n}$.
Hence
$$\sum_{V \in \mathcal{V}}\sup_{y\in V}\mathrm{e}^{f_{F_n}(y)}
=\sum_{V\in \mathcal{V}}\sup_{z\in
\pi^{-1}V}\mathrm{e}^{f_{F_n}\circ \pi (z)} \geq P_{F_n}(G,\mathcal
{F}\circ \pi;\pi^{-1}\mathcal {U}).$$ Since $\mathcal{V}$ is
arbitrary, we have that $P_{F_n}(G,\mathcal {F};\mathcal {U})\geq
P_{F_n}(G,\mathcal {F}\circ \pi;\pi^{-1}\mathcal{U})$.

Conversely, we note that
$$P_{F_n}(G,\mathcal{F}\circ
\pi;\pi^{-1}\mathcal{U})=\min_{\beta\in
\mathcal{P}^*((\pi^{-1}\mathcal{U})_{F_n})}\left \{\sum_{B\in
\beta}\sup_{x\in B}\mathrm{e}^{f_{F_n}(x)}\right \}.$$ Let $\beta\in
\mathcal{P}^*((\pi^{-1}\mathcal{U})_{F_n})$, then $\pi(\beta)\in
\mathcal{C}_Y,\ \pi(\beta)\succeq \mathcal{U}_{F_n}$, and thus
$$\sum_{B\in \beta}\sup_{x\in B}\mathrm{e}^{f_{F_n}\circ \pi(x)}
=\sum_{B\in \beta} \sup_{y\in \pi (B)}\mathrm{e}^{f_{F_n}(y)} \geq
P_{F_n}(G,\mathcal{F};\mathcal{U}).$$ Since $\beta$ is arbitrary,
$P_{F_n}(G,\mathcal {F}\circ \pi;\pi^{-1}\mathcal {U})\geq
P_{F_n}(G,\mathcal {F};\mathcal {U})$.

Above all, $P_{F_n}(G,\mathcal {F}\circ \pi;\pi^{-1}\circ\mathcal
{U})= P_{F_n}(G,\mathcal {F};\mathcal {U})$ for each $n\in
\mathbb{N}$, from which the lemma follows.
\end{proof}

For a fixed ${\mathcal U}=\{U_1, \cdots, U_M\} \in {\mathcal
C}_X^o$, we let
$$\mathcal{U}^*=\left\{\{A_1, \cdots, A_M\} \in {\mathcal P}_X:
A_m \subseteq U_m: 1 \leq m \leq M\right\}.$$ The following lemma
will be used in the computation of $H_{\mu}({\mathcal U})$ and
$h_{\mu}(T, \mathcal{U})$ (see \cite[Lemma 2]{HM} for detail).

\begin{lem}
Let $H: {\mathcal P}_X \rightarrow {\mathbb R}$ be monotone in the
sense that $H(\alpha) \geq H(\beta)$ whenever $\alpha \succeq
\beta$. Then
$$\inf_{\alpha\in \mathcal {P}_X,
\alpha\succeq \mathcal {U}}H(\alpha)=\inf_{\alpha\in \mathcal
{U}^{*}}H(\alpha).$$
\end{lem}

\begin{lem}
Let $(X, G)$ be $G$-system, where $G$ is a Abelian group. Suppose
$\{\nu_n\}_{n=1}^{\infty}$ is a sequence in ${\mathcal M}(X)$ and
$\{F_n\}_{n=1}^{\infty}$ is a tiling F$\phi$lner sequence of $G$. We
form the new sequence $\{\mu_n\}_{n=1}^{\infty}$ by
$\mu_n=\frac{1}{|F_n|}\sum_{g \in F_n}g\nu_n$. Assume that
$\mu_{n_i}$ converges to $\mu$ in ${\mathcal M}(X)$ for some
subsequence $\{n_i\}$ of natural numbers. Then $\mu \in {\mathcal
M}(X, G)$, and moreover
\begin{equation} \label{eq:5}
\limsup_{i \rightarrow \infty}\frac{1}{|F_{n_i}|}\int \,
f_{F_{n_i}}\, \mathrm{d}\nu_{n_i} \leq {\mathcal F}_*(\mu).
 \end{equation}
\end{lem}

\begin{proof}
The statement $\mu \in \mathcal{M}(X, G)$ is well-known. Now we show
the desired inequality. Fix $k \in {\mathbb N}$. Since $F_k$ is a
tile of $G$, let $C_n$ is a tiling center of $F_n$ relative to
$F_k$, i.e.,
\begin{equation} \label{eq:5-1}
\bigsqcup_{c\in C_n}F_kc \supseteq F_n \mbox{\ \ and \ \ } F_kc\cap
F_n\neq \emptyset,\ \forall~ c\in C_n, n \in {\mathbb N}.
 \end{equation}
By Lemma 3.1, for each $\epsilon>0$, when $n$ large enough, we have
\begin{equation} \label{eq:5-2}
|F_kC_n|\leq|(1+\epsilon)|F_n| \mbox{\ \ and\ \ } |F_kC_n\setminus
F_n|\leq \epsilon|F_n|.
 \end{equation}
Without loss of generality, we can assume $\mathcal {F}$ is
nonnegative monotone, then
$$f_{F_na}(x)\leq f_{\bigsqcup_{c\in C_n}F_kca}(x)
\leq \sum_{c\in C_n}f_{F_kca}(x),\ \ \ \forall a\in F_k.$$ Set
$g_{_{F_n}}(x)=\frac{1}{|F_k|}\sum_{a\in F_k}f_{F_na}(x)$. Since $G$
is a Abelian group, we have
$$g_{_{F_n}}(x)\leq \frac{1}{|F_k|}\sum_{a\in F_k,c\in C_n}f_{F_kca}(x)
=\frac{1}{|F_k|}\sum_{g\in F_kC_n}f_{F_kg}(x).$$ Moreover, by
\eqref{eq:5-2}, we can get
\begin{eqnarray*}
\frac{1}{|F_n|}\int_X g_{_{F_n}}(x)d\nu_n(x)
&\leq&\frac{1}{|F_n||F_k|}\int_X \sum_{g\in
F_kC_n}f_{F_kg}(x)d\nu_n(x)\\
&=&\frac{|F_kC_n|}{|F_n||F_k|}\int_X
f_{F_k}(x)d\widetilde{\mu}_n(x)\\
&\leq&\frac{1+\epsilon}{|F_k|}\int_X
f_{F_k}(x)d\widetilde{\mu}_n(x),
\end{eqnarray*}
where $\widetilde{\mu}_n=\frac{1}{|F_kC_n|}\sum_{g\in
F_kC_n}g\nu_n$. To complete the lemma, it suffices to show the
following two claims hold.

\vspace{2mm}

{\bf Claim 1.} $\lim \limits_{n \rightarrow
\infty}\displaystyle{\frac{1}{|F_n|}} \displaystyle{\int_X \,
\left|f_{F_n}(x)-g_{F_n}(x)\right|\, \mathrm{d}\nu_n(x)}=0$;

\vspace{2mm}

{\bf Proof of Claim 1.} Since $\F$ is nonnegative, monotone and
sub-additive, for each $a \in F_k$,
$$f_{F_n}(x) \leq f_{F_na}(x)+f_{F_n\setminus F_na}(x) \leq
f_{F_na}(x)+ K \cdot |F_n\setminus F_na|.$$ By symmetry,
$|f_{F_n}(x)-f_{F_na}(x)|\leq K \cdot |F_n\bigtriangleup F_na|$.
Thus,
$$
|f_{F_n}(x)-g_{_{F_n}}(x)|=\left|\frac{1}{|F_k|}\sum_{a\in
F_k}f_{F_n}(x)-f_{F_na}(x)\right| \leq \frac{K}{|F_k|}\sum_{a\in
F_k}|F_n\bigtriangleup F_na|.
$$
Therefore,
$$\frac{1}{|F_n|}\int_X \, |f_{F_n}(x)-g_{_{F_n}}(x)| \, \mathrm{d}\nu_n
\leq \frac{K}{|F_k|}\sum_{a\in F_k}\frac{|F_n\bigtriangleup
F_na|}{|F_n|}\rightarrow 0 \ \ (n\rightarrow \infty)\text{¡£}$$ This
complete the proof of Claim 1. \hspace*{\fill} $\square$

\vspace{2mm}

{\bf Claim 2.} with the weak$^*$-topology, $\widetilde{\mu}_{n_i}
\rightarrow \mu$.

\vspace{2mm}

{\bf Proof of Claim 2.} It suffices to show that for each $f\in
C(X)$,
\begin{equation} \label{eq:5-3}
\left| \int_X \, f \, \mathrm{d}\mu_n-\int_X \, f \,
\mathrm{d}\widetilde{\mu}_n \right| \rightarrow 0 \ \ (n\rightarrow
\infty)\text{¡£}
 \end{equation}
By Lemma 3.1,\ $\lim_{n \rightarrow \infty}
\frac{|F_kC_n|}{|F_n|}=1$. So
\begin{eqnarray*}
&& \lim_{n \rightarrow \infty}\left|\int_X \, f \, \mathrm{d}\mu_n-
\int_X \, f \, \mathrm{d}\widetilde{\mu}_n\right|\\
&=& \lim_{n \rightarrow \infty}\left|\frac{1}{|F_n|}\sum_{g\in
F_n}\int_X \, f(gx)\,
\mathrm{d}\nu_n(x)-\frac{1}{|F_kC_n|}\sum_{g\in
F_kC_n}\int_X \, f(gx) \, \mathrm{d}\nu_n(x)\right|\\
&=& \lim_{n \rightarrow \infty} \left|\frac{1}{|F_n|}\sum_{g\in
F_n}\int_X \, f(gx)\, \mathrm{d}\nu_n(x)-\frac{1}{|F_n|}\sum_{g\in
F_kC_n}\int_X \, f(gx) \, \mathrm{d}\nu_n(x)\right|\\
&\leqslant& \lim_{n \rightarrow \infty}\frac{|F_kC_n \setminus
F_n|}{|F_n|} \cdot \|f\|=0
\end{eqnarray*}
This complete the proof of Claim 2.
\end{proof}

The following lemma is well known (see \cite[$\S 9$]{Wal} for a
proof).

\begin{lem}
Let $a_1, a_2, \cdots, a_k$ be given real numbers. If $p_i \geq 0,
i=1, 2, \cdots, k$ and $\sum_{i=1}^k p_i=1$, then
\begin{equation}\label{eq:3}
\sum_{i=1}^k p_i \left(a_i-\log p_i\right) \leq \log
\left(\sum_{i=1}^k \mathrm{e}^{a_i}\right),
\end{equation}
and equality holds if and only if
$p_i=\frac{\mathrm{e}^{a_i}}{\sum_{i=j}^k \mathrm{e}^{a_j}}$ for all
$i=1, 2, \cdots, k$.
\end{lem}

\subsection{Proof of Theorem 1.1} In this section we give the proof
of Theorem 1.1.
\\
\\
{\it Proof of Theorem 1.1} We divide the proof into three small
steps:

\vspace{2mm}

{\bf Step 1.} $P(G, {\mathcal F};{\mathcal U}) \geq h_{\mu}(T,
{\mathcal U})+{\mathcal F}_*(\mu)$ for all $\mu \in {\mathcal M}(X,
G)$.

Let $\mu \in {\mathcal M}(X, G)$ and $\{F_n\}_{n \in {\mathbb N}}$ a
F$\phi$lner sequence of $G$. By \eqref{eq:2}, there exists a finite
partition $\beta \in {\mathcal P}^*({\mathcal U}_{F_n})$ such that
$$P_{F_n}(G, {\mathcal F}, {\mathcal U})=\sum_{B \in \beta}
\sup_{x \in B} \mathrm{e}^{f_{F_n}(x)}.$$ It follows from Lemma 3.7
that
\begin{eqnarray*}
\log P_{F_n}(G, {\mathcal F}, {\mathcal U})&=& \log\left(\sum_{B \in
\beta} \sup_{x \in B} \mathrm{e}^{f_{F_n}(x)}\right) \\
&\geq& \sum_{B \in \beta}\mu(B) \left(\sup_{x \in B}f_{F_n}(x)-\log
\mu(B)\right) \mbox{\ (by \eqref{eq:3})}\\
&=& H_{\mu}(\beta) + \sum_{B \in \beta}\sup_{x \in B}f_{F_n}(x)
\cdot \mu(B) \\
&\geqslant& H_{\mu}({\mathcal U}_{F_n}) + \int_X \, f_{F_n} \,
\mathrm{d}\mu
\end{eqnarray*}
The proof of step 1 is complete by dividing the above by $|F_n|$
then passing the limit $n \rightarrow \infty$.

\vspace{2mm}

{\bf Step 2.} If $(X, G)$ is a zero-dimensional $G$-system, then
there exists a $\mu \in \mathcal{M}(X, G)$ such that
\begin{equation} \label{eq:3-2-1}
P(G,\mathcal {F};\mathcal{U})\leq h_\mu(G,\mathcal {U})+\mathcal
 {F}_*(\mu).
\end{equation}

Let ${\mathcal U}=\{U_1, U_2, \cdots, U_d\}$ and define
$${\mathcal U}^*=\left\{\alpha \in {\mathcal P}_X: \alpha=\{A_1, A_2, \cdots, A_d\},
A_m \subseteq U_m, m=1, 2, \cdots, d\right\}.$$ Since $X$ is
zero-dimensional, the family of partitions in $\U^*$, which are
finer than $\mathcal{U}$ and consist of clopen (close and open)
sets, is countable. We let $\{\alpha_l\}_{l\in \mathbb{N}}$ denote
an enumeration of this family.

Let $\{F_n\}_{n \in {\mathbb N}}$ be a F$\phi$lner sequence of $G$
with $|F_n| \geq n$ for each $n \in {\mathbb N}$. By Lemma 3.3, for
each $n \in {\mathbb N}$, there exists a finite subset $B_n$ of $X$
such that
\begin{equation} \label{eq:7}
\sum_{x\in B_{F_n}}\mathrm{e}^{f_{F_n}(x)}\geq
\frac{{P_{F_n}(G,\mathcal {F};\mathcal {U})}}{n},
 \end{equation}
and each atom of $(\alpha_l)_{F_n}$ contains at most one point of
$B_n$, for each $l=1,\cdots,n$. Let
$$\nu_n=\sum_{x\in B_n}\lambda_n(x)\delta_x \mbox{\ \ \ and \ \ \ }
\mu_n=\frac{1}{|F_n|}\sum_{g\in F_n}g\nu_n,$$ where
$\lambda_n(x)=\frac{\mathrm{e}^{f_{F_n}(x)}}{\sum_{y\in
B_n}\mathrm{e}^{f_{F_n}(x)}}$ for $x \in B_n$. Since
$\mathcal{M}(X,G)$ is compact we can choose a subsequence $\{n_i\}
\subseteq {\mathbb N}$ such that $\mu_{n_i} \rightarrow \mu$ in the
weak$^*$-topology of $\mathcal{M}(X)$. It is easy to check $\mu \in
\mathcal{M}(X,G)$. We wish to show that $\mu$ satisfies
\eqref{eq:3-2-1}. By Lemma 3.5 and the fact that
$$h_\mu^{+}(G,\mathcal {U})
=\inf_{\beta \in \U^*}h_{\mu}(G,\beta) =\inf_{l\in
\mathbb{N}}h_{\mu}(G,\alpha_l),$$ it is sufficient to show that for
each $l\in \mathbb{N}$£¬
\begin{equation} \label{eq:8}
P(G,\mathcal {F};\mathcal{U})\leq h_\mu(G,\alpha_l)+\mathcal
 {F}_*(\mu).
 \end{equation}

Fix $l \in {\mathbb N}$. For each $n>l$, we know from the
construction of $B_n$ that each atom of $(\alpha_l)_{F_n}$ contains
at most one point of $B_n$, and
\begin{equation} \label{eq:9}
\sum_{x \in B_n}-\lambda_n(x)\log \lambda_n(x)=\sum_{x \in
B_n}-\nu_n(\{x\})\log \nu_n(\{x\})=H_{\nu_n}((\alpha_l)_{F_n}).
 \end{equation}
Moreover, it follows from \eqref{eq:7}, \eqref{eq:9} that
\begin{eqnarray*}
\log P_{F_n}(G,\mathcal {F};\mathcal {U})-\log
n&\leq&\log\left(\sum_{x\in
B_n}\mathrm{e}^{f_{F_n}(x)}\right)=\sum_{x\in
B_n}\lambda_n(x)(f_{F_n}(x)-\log\lambda_n(x))\\
&=&H_{\nu_n}((\alpha_l)_{F_n})+\sum_{x\in
B_n}\lambda_n(x)f_{F_n}(x)\\
&=&H_{\nu_n}((\alpha_l)_{F_n})+\int_X
f_{F_n}(x)\,\mathrm{d}\nu_n(x).
\end{eqnarray*}
Hence,
\begin{equation} \label{eq:10}
\log P_{F_n}(G,\mathcal {F};\mathcal{U})-\log n \leq
H_{\nu_n}((\alpha_l)_{F_n})+\int_X f_{F_n}(x)\,\mathrm{d}\nu_n(x)
 \end{equation}
Without loss of generality, we can assume ${\mathcal F}$ is
nonnegative, monotone and sub-additive.

\vspace{2mm}

{\it Case 1.} $G$ is Abelian group. We can assume $\{F_n\}_{n \in
\mathbb{N}}$ is a tiling F$\phi$lner sequence. Since $E\in
F(G)\mapsto H_{\nu_n}((\alpha_l)_E)$ is a nonnegative, monotone and
strongly sub-additive function, it follows from Lemma 3.2 that for
each $B \in F(G)$, one has
\begin{equation} \label{eq:11}
\begin{split}
\frac{1}{|F_n|}H_{\nu_n}((\alpha_l)_{F_n})&\leq
\frac{1}{|F_n|}\sum_{g\in
F_n}\frac{1}{|B|}H_{\nu_n}((\alpha_l)_{Bg})+\frac{|F_n\setminus
A_{F_n,B}|}{|F_n|}\cdot \log |\alpha_l|\\
&= \frac{1}{|B|}\frac{1}{|F_n|}\sum_{g\in
F_n}H_{g\nu_n}((\alpha_l)_{B})+\frac{|F_n\setminus
A_{F_n,B}|}{|F_n|}\cdot \log d\\
&\leq \frac{1}{|B|}H_{\mu_n}((\alpha_l)_{B})+\frac{|F_n\setminus
A_{F_n,B}|}{|F_n|}\cdot \log d.
 \end{split}
 \end{equation}
Set $B_1=B^{-1} \cup \{e_G\}$. Note that for each $\delta>0$, $F_n$
is $(B_1, \delta)$-invariant if $n$ is large enough and
$$F_n\setminus
A_{F_n,B}=F_n\cap B(G\setminus F_n)\subseteq
(B_1)^{-1}F_n\cap(B_1)^{-1}(G\setminus F_n)=B(F_n,B_1).$$ Letting $n
\rightarrow \infty$, we can get
\begin{equation} \label{eq:12}
\lim_{n\rightarrow\infty}\frac{|F_n\setminus A_{F_n,B}|}{|F_n|}\leq
\lim_{n\rightarrow\infty}\frac{|B(F_n,B_1)|}{|F_n|}=0.
\end{equation}
Hence, combining Lemma 3.6, \eqref{eq:10}, \eqref{eq:11} and
\eqref{eq:12} we obtain
\begin{eqnarray*}
P(G,\mathcal{F};\mathcal{U})&=&\lim_{i\rightarrow\infty}\frac{\log
P_{F_{n_i}}(G,\mathcal
{F};\mathcal{U})}{|F_{n_i}|}\\
&\leq&\limsup_{i\rightarrow\infty}\left(\frac{1}{|F_{n_i}|}H_{\nu_{n_i}}((\alpha_l)_{F_{n_i}})
+ \frac{\log n_i}{|F_{n_i}|}+\frac{1}{|F_{n_i}|}\int_X f_{F_{n_i}}(x) \,\mathrm{d}\nu_{n_i}(x)\right)\\
&\leq&\frac{1}{|B|}H_{\mu}((\alpha_l)_{B})+\mathcal{F}_*(\mu).
\end{eqnarray*}
By arbitrary of $B \in {\mathcal F}(G)$, \eqref{eq:8} holds.

\vspace{2mm}

{\it Case 2.} ${\mathcal F}$ is strongly sub-additive. Now $E \in
F(G)\mapsto \displaystyle{\int_X \, f_E(x) \,\mathrm{d}\nu_n(x)}$ is
a nonnegative monotone strongly sub-additive function. By Lemma 3.2,
for each $B \in {\mathcal F}(G)$, one has
\begin{equation} \label{eq:13}
\begin{split}
\frac{1}{|F_n|}\int_X f_{F_n}(x)\,\mathrm{d}\nu_n(x) &\leq
\frac{1}{|F_n|}\sum_{g\in F_n}\frac{1}{|B|}\int_X f_{Bg}(x)
\,\mathrm{d}\nu_n(x)+\frac{|F_n\setminus
A_{F_n,B}|}{|F_n|}\cdot K\\
&= \frac{1}{|B|}\int_X f_{B}(x)
\,\mathrm{d}\mu_n(x)+\frac{|F_n\setminus A_{F_n,B}|}{|F_n|}\cdot K
 \end{split}
 \end{equation}
Combining \eqref{eq:10}, \eqref{eq:11}, \eqref{eq:12} and
\eqref{eq:13}, we have
\begin{eqnarray*}
P(G,\mathcal{F};\mathcal{U})&=&\lim_{i \rightarrow\infty}\frac{\log
P_{F_{n_i}}(G,\mathcal
{F};\mathcal{U})}{|F_{n_i}|}\\
&\leq&
\limsup_{i\rightarrow\infty}\left(\frac{1}{|F_{n_i}|}H_{\nu_{n_i}}((\alpha_l)_{F_{n_i}})
+ \frac{\log n_i}{|F_{n_i}|}+\frac{1}{|F_{n_i}|}\int_X f_{F_{n_i}}(x) \,\mathrm{d}\nu_{n_i}(x)\right)\\
&\leq&\frac{1}{|B|}H_{\mu}((\alpha_l)_{B})+\frac{1}{|B|}\int_X
f_B(x) \,\mathrm{d}\mu(x)
\end{eqnarray*}
By arbitrary of $B \in {\mathcal F}(G)$ and (2) in Remark 2.2,
\eqref{eq:8} holds.

\vspace{2mm}

{\bf Step 3.} For $G$-system $(X, G)$, there exists a $\mu \in
\mathcal{M}(X, G)$ such that \eqref{eq:3-2-1} holds. It is well
known that there exists factor map $\pi: (Z, G) \rightarrow (X, G)$,
where $(Z, G)$ is a zero-dimensional $G$-systems (see \cite{HYZ},
for example). Using Step 2, there is $\nu \in \mathcal{M}(Z,G)$ such
that $$P(G,\mathcal{F}\circ\pi; \pi^{-1}(\U)) \leqslant h_\nu
(G,\pi^{-1}\mathcal {U})+(\mathcal
 {F}\circ\pi)_*(\nu).$$
Let $\mu=\pi_*(\nu)$. By Lemma 3.4, we can get
\begin{eqnarray*}
h_\mu (G,\mathcal {U})+\mathcal
 {F}_*(\mu)&=& \inf_{\alpha \in {\mathcal P}_X, \alpha \succeq {\mathcal
 U}}\left(h_{\mu}
(G, \alpha)+ \F*(\mu)\right)\\
&=& \inf_{\alpha\in {\mathcal P}_X, \alpha \succeq {\mathcal
 U}}\left(h_{\nu}
(G, \pi^{-1}(\alpha))+ (\F \circ \pi)*(\nu)\right)\\
&\geq& h_{\nu} (G, \pi^{-1}{\mathcal U})+ (\F \circ \pi)*(\nu)\\
&\geq& P(G,\mathcal {F}\circ\pi;\pi^{-1}(\U))= P(G,\mathcal
{F};\mathcal{U}).
\end{eqnarray*}

\vspace{2mm}

We will show the supremum of \eqref{eq:main} can be attained in
$\mathcal{M}^e(X, G)$. Let $\mu=\int_{\mathcal{M}^e(X,G)}\, \theta
\,\mathrm{d}m(\theta)$ be the ergodic decomposition of $\mu$. Note
that $\theta \mapsto {\mathcal F}_*(\theta)$ is $m$-measurable and
$\sup\{\frac{|f_{F_n}(x)|}{|F_n|}:x\in X, F_n\in {\mathcal
F}(G)\}<\infty$. Then, by Lebesgue's Dominated Convergence Theorem,
\begin{equation} \label{eq:14}
\begin{split}
\int_{\mathcal{M}^e(X,G)}\mathcal{F}_*(\theta) dm(\theta) &=
\int_{\mathcal{M}^e(X,G)}\lim_{n\rightarrow\infty}\frac{1}{|F_n|}\int_X
f_{F_n}(x)\,\mathrm{d}\theta(x)\,\mathrm{d}m(\theta)\\
&=
\lim_{n\rightarrow\infty}\int_{\mathcal{M}^e(X,G)}\frac{1}{|F_n|}\int_X
f_{F_n}(x)\,\mathrm{d}\theta(x)\,\mathrm{d}m(\theta)\\
&= \lim_{n\rightarrow\infty}\frac{1}{|F_n|}\int_X
f_{F_n}(x)\,\mathrm{d}\mu(x)=\mathcal{F}_*(\mu).
 \end{split}
 \end{equation}
Combining Lemma 2.6, \eqref{eq:14} and \eqref{eq:main},
\begin{eqnarray*}
P(G,\mathcal {F};\mathcal{U})&=& h_\mu(G,\mathcal {U})+ \mathcal
 {F}_*(\mu)\\
&=&\int_{\mathcal{M}^e(X,G)}h_{\theta}(G,\mathcal
{U})\,\mathrm{d}m(\theta)+\int_{\mathcal{M}^e(X,G)}\mathcal
 {F}_*(\theta) \,\mathrm{d}m(\theta)\\
&=&\int_{\mathcal{M}^e(X,G)}\left(h_{\theta}(G,\mathcal
{U})+\mathcal
 {F}_*(\theta)\right) \,\mathrm{d}m(\theta).
\end{eqnarray*}
Hence there exists $\theta \in {\mathcal M}^e(X, G)$ such that
$$P(G,\mathcal {F};\mathcal{U})\leq h_{\theta}(G,\mathcal {U})+\mathcal
 {F}_*(\theta),$$
which is complete the proof of Theorem 1.1.  \hspace*{\fill}
$\square$

\section{An Example}

In this section, we main give a nontrivial example of sub-additive
potentials that satisfies conditions (C1), (C2) and (C3).

\begin{ex}
Let $(X,G)$ be a $G$-system and $M : X \rightarrow \mathbb{R}^{n
\times n}$ a nonnegative continuous matrix function of $X$, i.e.,
$M=\left(M_{i,j}\right)_{n\times n}$, where $M_{i,j}$ are
nonnegative continuous of $X$ for all $i,j=1,2,\cdots,n$.

Now we define $\mathcal{F}=\{f_E\}_{E \in {\mathcal F}(G)}$ as
follows: for each $x \in X$ and $E \in {\mathcal F}(G)$,
$$f_E(x):=\min_{1 \leq m \leq |E|} \min_{(g_i) \in E^m} \log \left\|\prod_{i=1}^m M(g_ix)\right\|,$$
where $\|M(x)\|=\sum_{i,j=1}^n M_{i,j}(x)$. We will show that
$\mathcal{F}$ satisfies conditions (C1), (C2) and (C3)

\vspace{2mm}

$\bullet$ $f_{Eg}(x)=f_E(gx)$ for all $x \in X$, $g \in G$ and $E
\in {\mathcal F}(G)$;

\vspace{2mm}

$\bullet$ $f_{E \cup F} \leqslant f_E + f_F$ for all $E, F \in
{\mathcal F}(G)$ with $E \cap F=\emptyset$;

\vspace{2mm}

$\bullet$ Given $E \in {\mathcal F}(G)$ and $g \notin E$. Then there
exists $1 \leq i \leq m \leq |E|$ such that
$$f_{E\cup \{g\}}(x)=\log\|M(g_1x)\cdots
M(g_ix)M(gx)M(g_{i+1}x) \cdots M(g_mx)\|.$$ Set $A=M(g_1x)\cdots
M(g_ix)$, $B=M(gx)$ and $C=M(g_{i+1}x)\cdots M(g_mx)$. Thus,
$$f_E(x)-f_{E\cup
\{g\}}(x)\leq \log \frac{\|AC\|}{\|ABC\|}.$$ Let
$$K_1=\frac{\min_{1\leq i,j\leq n}\min_{x\in
X}M_{i,j}(x)}{\max_{1\leq i,j\leq n}\max_{x\in X}M_{i,j}(x)},\ \ \
K_2=\min_{1\leq i,j\leq n}\min_{x\in X}M_{i,j}(x).$$ Then $K_1, K_2
\in (0, +\infty)$ and $M(x)-\frac{K_1}{n}EM(x)$ is a nonnegative
matrix, where $E=(E_{i,j}), E_{i,j}\equiv 1$. Hence,
\begin{eqnarray*}
\|ABC\|&\geq&\left\|A\frac{K_1}{n}EBC \right\|
=\frac{K_1}{n}\left\|A \right\|\|BC\|\\
&\geq&\left(\frac{K_1}{n}\right)^2\|A\|\|BEC\|
=\left(\frac{K_1}{n}\right)^2\|A\|\|B\|\|C\|\\
&\geq&\left(\frac{K_1}{n}\right)^2n^2K_2\|A\|\|C\|=
K_1^2K_2\|A\|\|C\|,
\end{eqnarray*}
which follows that
$$\frac{\|AC\|}{\|ABC\|}\leq \frac{1}{K_1^2K_2}<\infty.$$
\end{ex}
\noindent
\\
\\
{\bf Acknowledgement.} We thank Professor Xiangdong Ye and Professor
Wen Huang for useful discussion and suggestions over the topic.
Particularly, we thank Ye for the careful reading and valuable
comments which greatly improved the writing of the paper.

\noindent
\\

\end{document}